\documentclass{article} 
\usepackage{caption2}

\makeatletter
\@addtoreset{footnote}{page}
\makeatother

\author{J Marshall Ash, Michael Ash, Rafael Ash, Ángel Plaza}

\usepackage{amsmath,amsthm}
\usepackage{amssymb}
\usepackage{amsfonts}
\theoremstyle{plain}
\newtheorem{theorem}{Theorem}
\theoremstyle{definition}

\newtheorem{conjecture}[theorem]{Conjecture}

\usepackage{graphicx}
\usepackage{float}
\usepackage{hyperref} 
\usepackage{url}
\usepackage[raggedright]{titlesec}

\usepackage{multicol}

\usepackage{pgfplots}
\pgfplotsset{compat=1.3}
\usepackage{graphicx}
\usepackage{tikz}
\usetikzlibrary{plotmarks,calc,math}

\begin{document}

\title{Iterated harmonic numbers}
\markright{Iterated harmonic numbers}

\maketitle

\begin{abstract}
 The harmonic numbers are the sequence $1, 1+1/2, 1+1/2+1/3, \cdots$.
  Their asymptotic difference from the sequence of the natural logarithm of the positive integers is Euler's constant gamma.
  We define a family of natural generalizations of the harmonic numbers.
  The $j$th iterated harmonic numbers are a sequence of rational numbers that nests the previous sequences and relates in a similar way to the sequence of the $j$th iterate of the natural logarithm of positive integers.
  The analogues of several well-known properties of the harmonic numbers also hold for the iterated harmonic numbers, including a generalization of Euler's constant.
  We reproduce the proof that only the first harmonic number is an integer and, providing some numeric evidence for the cases $j = 2$ and $j = 3$, conjecture that the same result holds for all iterated harmonic numbers.
 We also review another proposed generalization of harmonic numbers. 
\end{abstract}

\noindent
\section{Introduction: definitions and properties}\label{sec:introduction}

The harmonic numbers are the sequence $\left\{ 1,1+\frac{1}{2},1+\frac{1}{2}+\frac{1}{3},\cdots \right\} $; 
we denote them as $h_{1}\left( n\right) :=\sum_{k=1}^{n}\frac{1}{k}$.
Some of the properties of the harmonic numbers $h_{1}\left( n\right)$ are:

\begin{enumerate}
\item They are positive rational numbers, starting at $1$ and monotonically
increasing.

\item They are the partial sums of the divergent infinite series $\sum_{k=1}^{\infty }\frac{1}{k}$.

\item They have a direct connection with the natural logarithm, in
particular there is defined a constant $\gamma $ so that as $n\rightarrow \infty$, $h_{1}\left( n\right) -\ln n\rightarrow \gamma$.

\item They are close to similar, but convergent, sequences; $\sum_{k=1}^{\infty }\frac{1}{k}$ diverges, but, for any fixed small positive number $\epsilon$,  $\sum_{k=1}^{\infty}\frac{1}{k^{1+\epsilon }}$ converges.

\item The only harmonic number that is an integer is $1$.
\end{enumerate}

\section{\mbox{Iterated harmonic numbers and iterated logarithms}}

We define $h_{j}(n)$, the iterated harmonic numbers of order $j$, for $j=2,3,\cdots $. First define 
$h_{2}\left( n\right) :=\sum_{k=1}^{n}\frac{1}{kh_{1}\left( k\right) }$, then $h_{3}\left( n\right) :=\sum_{k=1}^{n}\frac{1}{kh_{1}\left( k\right) h_{2}\left( k\right) }$, and so on. 
For every integer $j\geq 2$, 
\begin{equation*}
h_{j}\left( n\right) :=\sum_{k=1}^{n}\frac{1}{kh_{1}\left( k\right) \label{eqn:iterated-h-sub-j}
h_{2}\left( k\right) \cdots h_{j-1}\left( k\right) }.
\end{equation*}

The natural logarithm is $\ln_1(x)=\ln x:=\int_{1}^{n}\frac{dt}{t}$.  The second iterated natural logarithm is
$\ln_2 {\left(x\right)} = \ln \ln x :=\int_{e}^{x}\frac{dt}{t\ln t}$.\footnote{This notation for the second iteration of the natural logarithm should
not be confused with $\log_2 n$, the base~$2$ logarithm of $n$. Only natural logarithms will be iterated in this work.} To see this, integrate
using the substitution $u= \ln t$. In general, the $j$th iterated logarithm is given inductively by

$$  \ln_{j}\left( x \right):= \int_{{}^ {j-1 }e}^x \frac{dt}{{t\ln_1 t\ln _{2}t\cdots \ln _{j-1}t  }}.  $$ \label{eqn:iterate-ln-sub-j}

To see this, use the substitution $u= \ln _{j-1}t, du=\frac{dt}{t \ln_1 t\ln _{2}t\cdots \ln _{j-2}t}$. In order to avoid a non-zero constant term
the lower limit of the integral must be ${}^{j-1}e$ where the constants  
$\{{}^ie\}$ are given by ${}^0e = 1,  {}^1e = e, {}^2e = e^e, {}^3e =e^{{}^2e} = e^{e^e},\dots. $
 For each $i = 0,1,2,\dots$, the expression ${}^ie$ is called the $i$th hyperpower of $e$.
 
  Let $(d_j,\infty)$ denote the domain of $\ln_j$. Then $d_1=0, d_2={}^0e=1, d_3={}^1e=e\approx 2.7,
 d_4={}^2e  = e^e \approx 15.2, d_5 = {}^3e \approx 3814279.1$, and $d_6={}^4e$. The number ${}^4e$ is too large to calculate,
 ${}^4e = 10^{(\log_{10}e)( {}^3e)} \approx {10}^{1656520}$ a number so large that it would take more than $1.6$ million decimal digits 
 just to write it down.

In order to draw a connection between the $j$th-iterated harmonic number and the $j$th-iterated logarithm, we also define a third quantity. It is the integer stepwise constant estimate for the integral defining the iterated logarithm:

 \begin{equation}
l_j{(n)} :=  \sum_{k=a}^n\frac{1}{k\ln k\ln_{2}k\cdots \ln _{j-1}k}.   \label{eqn:iterated-l-sub-j}
\end{equation}
The lower limit $a=a(j)$ is the smallest positive integer in the domain of the integral defining $\ln_j$. Thus $a(1)=1,a(2)=2$, 
and $a(j)=\lceil {}^{j-2}e \rceil$ for $j=3,4,\dots$.

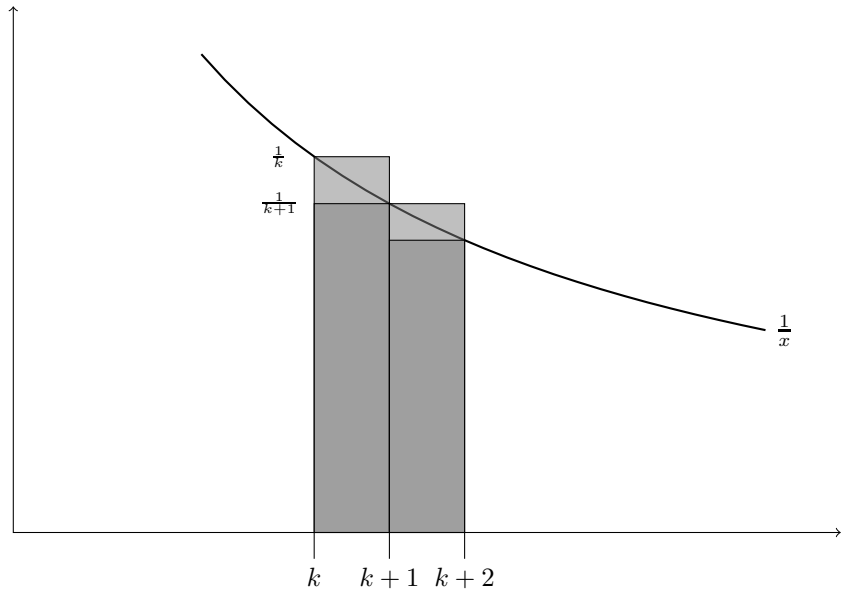
\begin{figure}[H]
\begin{center}
\begin{tikzpicture}[domain=5.5:13,yscale=35]
  \draw[color=black, thick]   plot (\x,{1/\x })  node[right] {$\frac{1}{x}$};

  \draw[<->] (3,0.2) -- (3,0) -- (14,0) ;

  \tikzmath{
    \starting = 7 ;
    \aln = 1 / (\starting) ;
    \bln = 1 / (\starting+1) ;
    \cln = 1 / (\starting+2) ;
  }

  \draw (\starting,0) -- (\starting,-0.01) node [black,below] {$k$} ;
  \draw (\starting+1,0) -- (\starting+1,-0.01) node [black,below] {$k+1$} ;
  \draw (\starting+2,0) -- (\starting+2,-0.01) node [black,below] {$k+2$} ;

  \draw [fill=gray, fill opacity=0.5, draw opacity=1.0] (\starting+1,0) rectangle (\starting,\aln)  node[align=center, text width=0.5cm, above=0.0cm, left=0.1cm, opacity=1.0] {\tiny $\frac{1}{k}$} ;
  \draw [fill=gray, fill opacity=0.5, draw opacity=1.0] (\starting+1,0) rectangle (\starting,\bln) node[align=center, text width=0.5cm, above=0.0cm, left=0.1cm, opacity=1.0] {\tiny $\frac{1}{k+1}$} ;

  \draw [fill=gray, fill opacity=0.5, draw opacity=1.0] (\starting+1,0) rectangle (\starting+2,\bln) ;
  \draw [fill=gray, fill opacity=0.5, draw opacity=1.0] (\starting+1,0) rectangle (\starting+2,\cln) ;

\end{tikzpicture}
\caption{Upper and lower step function bounds for $\frac{1}{x}$.}
\label{fig:one-over-x}
\end{center}
\end{figure}

Identify a sequence $\{ S(k) \}$ with a step function according to the rule that the function take the constant value $S(k)$ at every point $x$ of the interval $k \le x < k+1$. Figure 1 illustrates property 3, the connection between harmonic numbers and natural logarithm.  The area  $\int_1^n \frac{dx}{x} = \ln x$ under the function $\frac{1}{x}$ is greater than the darkly shaded area under the step function $\{\frac{1}{2}, \frac{1}{3}, \ldots , \frac{1}{n} \}$  which evaluates to $h_1(n) -1$.  On the other hand the area $\int_1^n \frac{dx}{x}$ is less than the area under the other step function $\{1, \frac{1}{2}, \ldots, \frac{1}{n-1} \}$. This is the total shaded area which evaluates to  $1 + \frac{1}{2} + \cdots + \frac{1}{n-1} = h_1(n-1)$  Analysis of the small differences between the three regions leads to the problem of accurately quantifying the difference between the harmonic numbers $h_1(n)$ and $\ln n$. 

For any higher order iteration $j$, the geometry is similar. The full complexity is already visible for $j=2$, which is illustrated in Figure~2.
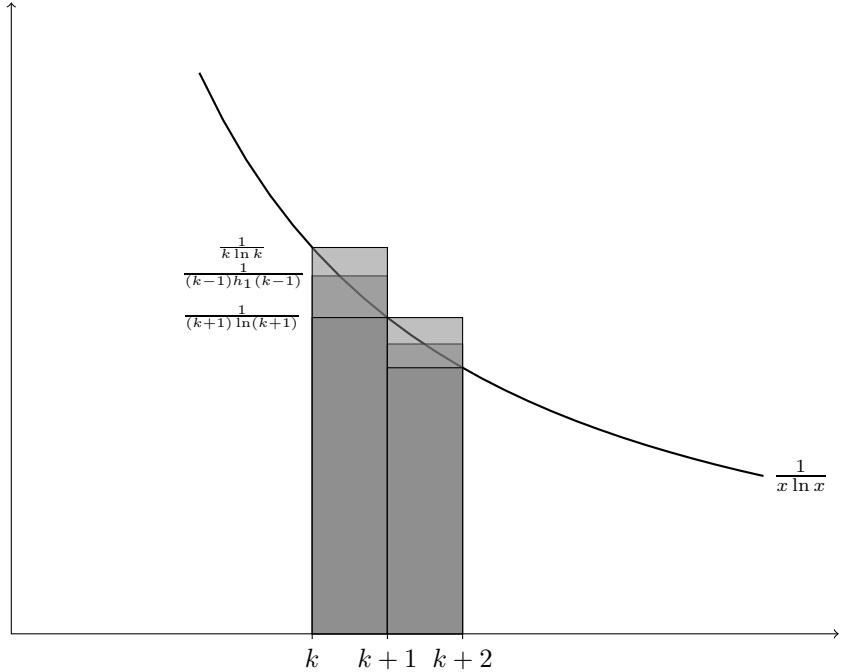
\begin{figure}[H]
\begin{center}
  \begin{tikzpicture}[domain=5.5:13,yscale=70]

    \tikzmath{
      \starting = 7 ;
      \aln = 1 / (\starting*ln(\starting)) ;
      \bln = 1 / ((\starting+1)*ln(\starting + 1)) ;
      \cln = 1 / ((\starting+2)*ln(\starting + 2)) ;
      \zh  = 1 / (6*(1 + 1/2 + 1/3 + 1/4 + 1/5 + 1/6)) ;
      \ah  = 1 / (7*(1 + 1/2 + 1/3 + 1/4 + 1/5 + 1/6 + 1/7)) ;
      \bh  = 1 / (8*(1 + 1/2 + 1/3 + 1/4 + 1/5 + 1/6 + 1/7 + 1/8)) ;
      \ch  = 1 / (9*(1 + 1/2 + 1/3 + 1/4 + 1/5 + 1/6 + 1/7 + 1/8 + 1/9)) ;
    }

    \draw[<->] (3,0.12) -- (3,0) -- (14,0) ;

    \draw[color=black, thick]   plot (\x,{ 1 / ( (\x) * ln (\x) )  })  node[right] {$\frac{1}{x \ln x}$};
    
    \draw (\starting,0) -- (\starting,-0.001) node [black,below] {$k$} ;
    \draw (\starting+1,0) -- (\starting+1,-0.001) node [black,below] {$k+1$} ;
    \draw (\starting+2,0) -- (\starting+2,-0.001) node [black,below] {$k+2$} ;

  \draw [fill=gray, fill opacity=0.5, draw opacity=1.0] (\starting,\zh) node[align=center, text width=1.6cm, above=0.0cm, left=0.01cm, opacity=1.0] {\tiny $\frac{1}{(k-1) h_1(k-1)}$} rectangle (\starting+1,0)  ; 

  \draw [fill=gray, fill opacity=0.5, draw opacity=1.0] (\starting,\aln) node[align=center, text width=1.6cm, above=0.0cm, left=0.01cm, opacity=1.0] {\tiny $\frac{1}{k \ln k}$}         rectangle (\starting+1,0)  ;
  \draw [fill=gray, fill opacity=0.5, draw opacity=1.0] (\starting,\bln) node[align=center, text width=1.6cm, above=0.0cm, left=0.01cm, opacity=1.0] {\tiny $\frac{1}{(k+1) \ln (k+1)}$} rectangle (\starting+1,0)  ;
  \draw [fill=gray, fill opacity=0.5, draw opacity=1.0] (\starting+2,\ah)       rectangle (\starting+1,0)  ; 

  \draw [fill=gray, fill opacity=0.5, draw opacity=1.0] (\starting+1,0) rectangle (\starting+2,\bln) ; 
  \draw [fill=gray, fill opacity=0.5, draw opacity=1.0] (\starting+1,0) rectangle (\starting+2,\cln) ; 


\end{tikzpicture}
\caption{Curve $\frac{1}{x \ln x}$ and step functions  $ \left\{ \frac{1}{k \ln k} \right\} $, $ \left\{ \frac{1}{(k-1) h_1(k-1)} \right\} $, and $ \left\{ \frac{1}{(k+1) \ln (k+1)} \right\} $.}
\label{fig:one-over-x-log-x}
\end{center}
\end{figure}

The area $\int_2^n \frac{dx}{x \ln x} \approx \ln_2 x + 0.37\ldots$ under the function $\frac{1}{x \ln x}$ is greater than the darkly shaded area under the step function $\left\{ \frac{1}{3 \ln 3}, \frac{1}{4 \ln 4}, \cdots, \frac{1}{n \ln n} \right\}$, which evaluates to $\ln_2 (n) - \frac{1}{2 \ln 2} $.  On the other hand, the area $\int_2^n \frac{dx}{x \ln x}$ is less than the area under the other step function $\left\{\frac{1}{2 \ln 2}, \frac{1}{3 \ln 3}, \frac{1}{4 \ln 4}, \cdots, \frac{1}{(n-1) \ln (n-1)} \right\}$. This last term is the total shaded area which evaluates to $\frac{1}{2 \ln 2} + \frac{1}{3 \ln 3} + \cdots + \frac{1}{(n-1) \ln (n-1)} = \l_2(n-1)$. Notice that $\l_1(n)$ coincides with $h_1(n)$, which is close to $\ln(x)$.  Similarly $\l_2(x)$ is close to $\ln_2(x)$. When we quantify this, we will naturally arrive at a constant. We can do a similar operation for each $j \ge 3$, which will be discussed in Section~\ref{sec:asymptotic-estimates}.

Look at the segment labeled $\frac{1}{(k-1)h_1(k-1)}$. Numerical calculations indicate that this segment's height satisfies
\begin{equation}
\frac{1}{(k+1) \ln (k+1)} < \frac{1}{(k-1) h_1 (k-1)} < \frac{1}{k \ln k}
\label{eq:inequality}
\end{equation}
so that it partitions the naturally occurring rectangle having upper left corner $\left(k, \frac{1}{k \ln k} \right)$ and lower right corner $ \left(k+1, \frac{1}{(k+1) \ln (k+1)} \right)$ into an upper lightly-shaded rectangle and a lower rectangle whose shading is of intermediate gray-scale intensity.

Form a third step function from the heights of these intermediate segments. Its area is the union of the dark gray and intermediate gray segments; from the inequality above we see that
\[ h_2(n-1) = \frac{1}{1 \cdot h_1(1) } + \frac{1}{2 \cdot h_1(2)} + \cdots + \frac{1}{(n-2) \cdot h_1(n-2)} \] is caught between the upper and lower estimates of $\int_2^n \frac{dx}{x \ln x} \approx \ln_2 x + 0.37\ldots$.  This indicates that the second iterated harmonic numbers are close to the second iterated natural logarithm.


The well-known Euler's constant $\gamma$ expresses the asymptotic difference between natural logarithm and the harmonic numbers:
\begin{equation}
 \gamma = \lim_{n\to\infty} h_1(n) - \ln(n)= \lim_{n\to\infty} l_1(n) - \ln(n) \approx 0.577\ldots, \label{eqn:gamma}
\end{equation}
noting that the definitions of $h_j$ and $l_j$ are not distinct for $j=1$.  In generalizing $\gamma$ for the $j$th iteration, we need to distinguish between $h_j$ and $l_j$ for $j \ge 2$.
We will denote the version defined with $l_j$ as  $\gamma_{j}$ satisfying
\begin{equation}
 \lim_{n\to\infty} l_j{(n)}-\ln_jn=\gamma_j. \label{eqn:gamma-j}
\end{equation}
while the version defined with $h_j$ will be $\gamma_j^{\prime}$ satisfying
\begin{equation}
 \lim_{n\to\infty} h_j{(n)}-l_j{(n)}=\gamma_j^{\prime}, \label{eqn:gamma-j-prime}
\end{equation}
with proof that the limits exist provided in the next section. 

The discussion of the $l_{2} {(n)}$ part of Figure 2 looks a lot like the discussion of  $h_{1}(n)$ in Figure 1. For this reason, the value of any $\gamma_{j}$ is easy to find. However, the discussion of the $h_{2}{(n)}$ part of Figure 2 hints that the problem of finding even $\gamma_{2}^{\prime}$ might not be so easy.  This is what happens. See the section on asymptotic estimates below for more details.

The motivation for Property 4 and its generalization is explained in  \cite{ash-iterated-logarithms-2009}. The infinite series $\sum_{a<k<\infty }\frac{1}{k\ln k\ln _{2}k\cdots \ln_{j-1}k}$ is divergent since by relation (1) its partial sums increase boundlessly. However, for each $\epsilon >0$, the series 
\begin{equation*}
\sum_{a<k<\infty }\frac{1}{k\ln k\ln _{2}k\cdots \ln _{j-2}k\left( \ln _{j-1}k\right) ^{1+\epsilon }}
\end{equation*}
converges. Similar results also hold for iterated harmonic numbers. The infinite series 
\begin{equation*}
\sum_{k=1}^{\infty }\frac{1}{kh_{1}\left( k\right) h_{2}\left(
k\right) \cdots h_{j-1}\left( k\right) }
\end{equation*}
is divergent by relation (2). The integral test shows that for each
$\epsilon >0$, the series
\begin{equation*}
\sum_{k=1}^{\infty }\frac{1}{kh_{1}\left( k\right) h_{2}\left(
k\right) \cdots h_{j-2}\left( k\right) \left( h_{j-1}\left( k\right) \right)
^{1+\epsilon }}
\end{equation*}
converges. These pairs of infinite series are closer to each other as $j$ increases in the sense that the divergent one diverges more slowly while the other one converges more slowly. The logarithmic examples do not have rational partial sums, but the divergent harmonic sums do. If we set $\epsilon =1$, then the convergent harmonic series also have rational partial sums.

Reference \cite{ash-plaza-2021} contrasts the convergence of $\sum_{k=1}^{\infty }\frac{1}{k\left( h_{1}\left( k\right) \right) ^{1+\epsilon }}$ with the divergence of $\sum_{k=1}^{\infty}\frac{1}{kh_{1}\left( k\right) }$. 
The $n$th partial sum of $\sum_{k=1}^{\infty}\frac{1}{k\left( h_{1}\left( k\right) \right)}$ is $h_2(n)$, and its study led us to write the paper.

\section{\mbox{Asymptotic estimates for iterated harmonic numbers}}\label{sec:asymptotic-estimates}

Recall $l_{j}\left( n\right) $ was defined by equation \eqref{eqn:iterated-l-sub-j} to be equal to $\sum_{a<k\leq n}\frac{1}{k\ln k\ln _{2}k\cdots \ln _{j-1}k}$ for an appropriate constant $a$.

\begin{theorem}
\label{th:gamma-j-prime-ok}Notice that $l_{1}\left( n\right) =h_{1}\left( n\right) $ for $n\geq 1$. For each integer $j\geq 2$, there is a constant $\gamma_{j}$ such that
\begin{equation}
    l_{j}(n)-\ln _{j}\left( n\right) =
    \gamma _{j} +
    \frac{1}{2n\ln n\ln_{2}n\cdots \ln _{j-1}n}+O\left( \frac{1}{n^{2}\ln n\ln _{2}n\cdots \ln_{j-1}n}\right) . \label{eqn:estimate-gamma-prime}
\end{equation}
\end{theorem}

\begin{proof}
The function $f\left( x\right) =\frac{1}{x\ln x\ln _{2}x\cdots \ln _{j-1}x}$ has antiderivative $\ln _{j}\left( x\right) $ and satisfies  $f^{\prime \prime }\left( n\right) =O\left( \frac{1}{n^{3}\ln n\ln _{2}n\cdots \ln_{j-1}n}\right) $. Apply the Euler summation formula to $f$. See exercises 2 and 6 of section 15.23 of \cite{apostol-calculusii-1969}.
\end{proof}

Once $\gamma$ had been discovered, the next step was finding accurate approximations for its numerical value. Euler himself found the values of the first sixteen significant figures of its decimal expansion in 1736. More than a hundred million of its decimal places had been filled in by 2004 \cite{Gourdon-Sebah}.

We write $f\left( n\right) \asymp g\left( n\right) $ to mean that there are
two positive constants $a$ and $b$ so that $af\left( n\right) <g\left(
n\right) <bf\left( n\right) $; in words, $f$ is of the same order as $g$ 
\cite[page 7]{hardy-wright-number-1960}.

Abbreviate $h_1$ to $h$. Here are three approaches for getting an estimate for $\gamma$:

Minimal method: Prove that $\lim_{n\to\infty}(h(n) - \ln n)$ exists and call it $\gamma$. Then choose any large integer for $n$ and calculate ${h(n) - \ln n}$ .

Standard method: Prove that $\gamma$ exists and establish $h\left( n\right) -\ln n-\gamma \asymp \frac{1}{n}$. Rearrange this, moving $\gamma$ to the right side of the relation, shows us that  $h\left( n\right) -\ln n$ tends to $\gamma$ with an error on the order of $\frac{1}{n}$.

Improved method: Prove that  $\gamma$ exists and establish $h\left( n\right) -\ln n-\gamma-  \frac{1}{2n}\asymp \frac{1}{n^2}$. As in the standard method, isolate $h\left( n\right) -\ln n$ on the left hand side of the relation and thereby observe that the speed of approach to $\gamma$ is now improved to be of the order of  $\frac{1}{n^2}$. 

To illustrate the virtue of the improved method, we use the number $5$ as a ``large'' integer and attempt to estimate  $\gamma ={.577\dots}$. The standard method estimate is 
$h\left(5\right) -\ln 5 = 1+\frac{1}{2} \cdots+\frac{1}{5}-\ln 5 = .674\dots$ and the improved method estimate is ${.674\dots}-\frac{1}{2\cdot5} ={.674\dots}-.1=.574\dots$.




\begin{theorem}
\label{th:gamma-j-not-ok}
 For each integer $j\geq 2$, $h_{j}\left( n\right) $ tends to $\infty $ at the same rate as the $j$th
iterated logarithm. This is an immediate consequence of the relations that for every $j=2,3,\dots $, there are
constant $\gamma _{j}^{\prime}$ such that
\begin{equation}
h_{j}\left( n\right) - l_{j}\left(n\right) \asymp  \gamma _{j}^{\prime}+ \frac{1}{\ln _{j-1}n}.
\label{eqn:generalized-eulers-constant}
\end{equation}
\end{theorem}

For $j \geq 2$, we made an arbitrary choice of $^{j-1}e$ for the lower limit of integration. Any value in the interval $\left( ^{j-1}e,\infty \right)$ would also have worked. For example,
when $j = 3$ we chose the domain of integration to be $\left[^{2}e,n\right] \approx \left[15.15,n\right] $, but $\left[ 3,n\right]$ would also have worked since $3 > {}^{1}e = e$. The value
of the generalized Euler constant $\gamma_j$ for each $j \geq 2$ depends on the choice of the corresponding lower limit.

On the positive side, this theorem establishes the existence of the $\gamma_j^{\prime}$ for all $j=2,3,\cdots$. But on the negative side, this theorem  proves that the natural extension
 of Property 3 converges so slowly that it can not provide an efficient way to compute the numerical value of $\gamma_j^{\prime}$ when $j\geq 2$.
 
The finding in Theorem~\ref{th:gamma-j-prime-ok} that the error in  relation~\eqref{eqn:estimate-gamma-prime} is smaller than $O(\frac{1}{n^2})$ implies that the left side of that relation well approximates the $\gamma_j$. Thus to prove Theorem~\ref{th:gamma-j-not-ok} it is suffcient to examine only $h_j(n)-l_j(n)$ instead of  $h_j(n)-\ln_j(n)$. Here is a sketch of the proof when $j=2$. Fix $n$ and write
 
 $$h_2(n)-l_2(n) =  \sum_{k=2}^{n } \frac{1}{kh_1(k)} - \frac{1}{k\ln(k)}. $$
 
The $k$th term of the sum is asymptotically close to $-\frac{1}{k(\ln k)^2}$. The integral test shows that the series $\sum_{k=2}^\infty \frac{1}{k(\ln k)^2}$ converges and also that its $n$th partial tail is on the order of $\frac{1}{\ln n}$ so that the difference between the sum and its  $n$th partail sum is of the order of $\frac{1}{\ln n}$.
 
To do the proof for general $j$, change the denominator from $k{\ln k}{\ln_{2}  k}\cdots{\ln_{j-1} k}$ into $k{h_1(k)}{h_{2}( k)}\cdots{h_{j-1}(k)}$ in $j-1$ steps, each step changing an iterated logarithm into a corresponding iterated harmonic number,  and measure the effect of each step. This is a little bit complicated and not too interesting.

\section{\mbox{The \emph{\lowercase{p}}-adic valuation and a proof of Property~5}}
\label{sec:valuations-prop5}

We now define $p$-adic valuations and use $2$-adic valuations to prove Property 5. Then we study empirically the conjecture that Property 5 extends to $h_{j}\ $for every $j$.

Fix a prime number $p$.  We define the $p$-adic valuation $\nu_p$ on the field of rational numbers by $\nu_p (0)=-\infty$ and for all other rationals $a/b$ we define $\nu_p$ to be $r$ where $a/b = p^r\cdot a'/b'$ with $a',b'$ coprime to $p$.\footnote{The  literature splits on defining \emph{valuation} either: as the  $p$-adic \emph{norm}, $\left|p^r\cdot a'/b'\right|_p = p^{-r}$   except $\left| 0 \right|_p = 0$; or as the convention we follow,   $\nu_p(p^r\cdot a'/b') = r$ except $\nu_p(0)=-\infty$. The norm   definition satisfies the triangle inequality,  $\left| x  y \right|_p \le \left| x \right|_p  \left| y   \right|_p$, and the strong triangle inequality,   $\left| x  y \right|_p \le \max(\left| x \right|_p , \left| y   \right|_p)$. The equality case of the strong triangle inequality is   equivalent to equation~\eqref{eqn:strongtriangle}.}
Examples: $$\nu_2 (8)=3, \nu_2 (1/8)=\nu_2 (-5/8)=\nu_2 (-5/56)=\nu_2 (10/112)=-3.$$

Notice in particular that the $p$-valuation of a fraction remains the same whether or not that fraction is reduced.  We will make use of an important equality satisfied by every $\nu_p$. It asserts that 
\begin{equation}
\nu_p(x+y)= \min\left \{\nu_p(x),\nu_p(y)\right\} \textrm{ when } \nu_p(x) \neq \nu_p(y).  \label{eqn:strongtriangle}
\end{equation}

We generalize this to
\begin{align}
\nu_p(x_1+x_2+\cdots+x_n)= \min\left \{\nu_p(x_1),\ldots,\nu_p(x_n)\right\} \nonumber \\ \textrm{when the minimum occurs exactly once}.  \label{eqn:strongtriangleapplication}
\end{align}
To see why this is true, we write out the $n=3$ case. Without loss of generality, suppose $x_i=p^{r_i}\frac{a_i}{b_i}, i=1,2,3$ with all $a_i,b_i$ coprime with $p,r_2>r_1$, and $r_3>r_1$. Compute

$$x_1+x_2+x_3=p^{r_1}\frac{a_1}{b_1}+p^{r_2}\frac{a_2}{b_2}+p^{r_3}\frac{a_3}{b_3}=$$
$$p^{r_1}\frac{a_1b_2b_3+p^{r_2-r_1}a_2b_1b_3+p^{r_3-r_1}a_3b_1b_2}{b_1b_2b_3}.$$

The denominator is coprime to $p$. The numerator is congruent mod $p$ to $a_1b_2b_3$ and hence is also coprime with $p$.  So $\nu_p(x_1+x_2+x_3)=r_1$.

Property~5 was proved by Theisinger in 1915 \cite{theisinger-bemerkung-1915}. Conway and Guy \cite{conway-guy-1996} present his proof thus: Look at the term [of $h_1(n)$] with the highest power of $2$ in it. It has nothing with which to pair. So $h_1(2),h_1(3),h_1(4),\cdots$ have odd numerator and even denominator.

This proof is correct but insufficiently explicit. We now give our version using the language of valuations.

Note that a rational number with a negative $2$-valuation cannot be an integer. More precisely, if a rational number $a/b$ is such that 
$\nu_2 (a/b) =- r < 0$ then $b$  is divisible by the even integer $2^r$. Thus the absolute value of the denominator remains at least $2$, even after writing $a/b$ in reduced form, and $a/b$ is not an integer. So to prove Property~5, we need only show $\nu_2 (h_1(n)) < 0$ for all $n \geq 2$. 

Fix $n\geq 2$. Choose $r$ maximal so that $2^{r}\leq n$.  We note that $\nu_2(\frac{1}{2^r})=-r$. Also by equation \eqref{eqn:strongtriangleapplication},
$$\nu_2(h_1(n))=\nu_2(1+\frac{1}{2}+\frac{1}{3}+\cdots+\frac{1}{2^r}+\cdots+\frac{1}{n})=-r.$$

This holds since for each positive integer $k=2^\rho \sigma$ with $\rho$ an integer and $\sigma$ an odd integer such that $k\in [1,n]$ and  $k\neq {2^r}$, the definition of $r$ forces $\rho$ to be less than $r$. So $\nu_2(\frac{1}{k})= -\rho >-r;$ the unique term where the minimum $-r$ occurs is $ \frac{1}{2^r}.$

While the Theisinger proof of Property~5 requires only arithmetic, we also provide a simple and elegant proof by K{\"u}rschák \cite{Kurschak} that invokes a number theory result.  Bertrand's Postulate, i.e., there is always a prime $p$ in $\left( {\lfloor {n/2} \rfloor,n} \right]$ (see \cite{AZ} for a proof), implies that $h_1(n) = \sum_1^n \frac{1}{k} = 1  + \frac{1}{2} + \cdots + \frac{1}{\lfloor n/2 \rfloor} + \cdots + \frac{1}{p} + \cdots + \frac{1}{n} = \frac{1}{p}  \frac{a}{b}$ (with $a,b$ both integers) is never an integer.  Every denominator in the sum, except $p$ itself, is relatively prime to $p$. (Other than $p$, the smallest number not relatively prime to $p$ is $2p$, which is too large to be among the denominators ending at $n$.)  The denominator $b$ of the sum of fractions with denominators relatively prime to $p$ is also relatively prime to $p$. Thus $ h_1(n) =\frac{1}{p}  \frac{a}{b}$  is not an integer, since marching from $1/p$ in $a$ steps of length $1/b$ cannot reach an integer. To formalize this, change scale by multiplying everything by $b$. We are now marching from $b/p$ to $bh_1(n)$ in $a$ steps of length $1$. Since $b$ and $p$ are relatively prime, $b/p$ is not an integer, and so neither is its translation $bh_1(n)$ nor is $h_1(n)$.

\section{\mbox{A conjecture for \emph{\lowercase{j}}-times iterated harmonic numbers}}\label{sec:conjectures}

From the definitions, it is immediate that for each $j\geq 1$, there holds the identity 
\begin{equation*}
h_{j}\left( 1\right) =1.
\end{equation*}

\begin{conjecture}\label{conjecture:no-integers}
For each integer $j\geq 2$, the only $j$th iterated harmonic number that is an integer is $h_{j}\left( 1\right) $.
\end{conjecture}

When $j=1$, the statement of the conjecture becomes Property~5, which was proved above. The proof relies on the fact that for all $n\geq 2$, reduced $h_1(n)$ has an even denominator and so cannot be an integer. Demonstrating that the denominator of $h_j(n)$ is even would prove the conjecture. We present numerical data for the cases of $j=3$ and $j=2$. Our computations indicate that the even-denominator property holds for $h_3$. The case of $h_2$ is different, but our evidence, presented in Table~1, suggests a different argument can be found here also.

In computing numerical evidence, we restrict interest to the cases of $j=2$ and $j=3$. Write every rational number $h_{j}\left( n\right) $, $j=2,3$; $n=2,3,\dots $ as a reduced ratio of positive integers, $\frac{n_{j}\left( n\right) }{d_{j}\left( n\right) }$. It suffices to prove that each such denominator $d=d_{j}\left( n\right) $ is at least $2$. All the evidence we have accumulated points to the conjecture's truth. The rough argument is that the denominators of both $h_{2}\left( n\right) $ and $h_{3}\left( n\right) $ seem to grow steadily larger as $n$ increases, whereas the failure of the conjecture would have a denominator of 1 occurring. The proof for the $h_{1}$ case involved showing that all denominators were even and hence not equal to $1$. The evidence for $h_{3}$ does not rule out such a proof. The evidence for $h_{2}$ rules out an even denominator proof. Nevertheless it suggests many other possible proofs. If positive integer $d$ is divisible by prime $p$ then the $p$-valuation of $d$ is $\max \left\{ \nu :p^{\nu }|d\right\}$.  The $p$-valuation of $d$ is zero when $p$ does not divide $d$.  Notice that the proof of Property 5 given above shows that the $2$-valuation of the denominator of $h_{1}\left( n\right)$ is $\left\lfloor \log_2 n\right\rfloor$, where $ \log_2 n$ denotes the base-$2$ logarithm of $n$, the function inverse to base-$2$ exponentiation.

\textbf{Numerical evidence when }$j=3:$ We computed the $2$-valuation of the (reduced) denominator of $h_{3}\left( n\right) $ for $n=1$ to $2,000$. For $n=1$, the $2$-valuation is, of course,~$0$. For $n=2$ to $n=11$ the $2$-valuation is $2$. For $n=12$, the $2$-valuation is, shockingly,~$0$. (Nonetheless, $h_3(12)$ is not an integer.)  From $n=13$ through $n=31$ the $2$-valuation is $3$. From $n=32$ to $n=2,000$, the $2$-valuation is always equal to $6$.

\textbf{Numerical evidence when }$j=2:$ We attempted to calculate the $p$-valuation for the denominator of $h_{2}\left( n\right) $ for $n=1$ to $n=40,000$ for all 46 primes less than $200$. We ceased computation when we hit computer system limits at $n=27,477$. The $2$-valuation is always $0$. The $3$-valuation is $1$ from $n=2$ to $n=53$. The $3$-valuation then alternates irregularly between 0 and 1. The left panel of Table~\ref{tab:p-vals-of-H2} shows the variation in the 3-valuation of $h_{2}\left( n\right) $ up to $n=27,477$.

\begin{table}[h]\small
\begin{center}
  \centering
\begin{tabular}{rc|rc}
\hline\hline 
\multicolumn{2}{c|}{3-valuation of $\mbox{denom}(h_2\left(n\right))$} & 
\multicolumn{2}{c}{97-valuation of $\mbox{denom}(h_2\left(n\right))$} \\ \hline
$n$           & 3-valuation & $n$           & 97-valuation \\ \hline\hline
1             & 0           & 1--10         & 0            \\ 
2--53         & 1           & 11--95        & 1            \\ 
54--62        & 0           & 96--9,322     & 2            \\ 
63--65        & 1           & 9,323--9,407  & 1            \\ 
66--161       & 0           & 9,408--27,477 & 0            \\ 
162--188      & 1           &               &              \\ 
189--197      & 0           &               &              \\ 
198--1,457    & 1           &               &              \\ 
1,458--1,700  & 0           &               &              \\ 
1,701--1,781  & 1           &               &              \\ 
1,782--4,373  & 0           &               &              \\ 
4,372--5,102  & 1           &               &              \\ 
5,103--5,345  & 0           &               &              \\ 
5,346--27,477 & 1           &               & 
\end{tabular}
\end{center}
  \caption{Selected $p$-valuations of the denominator of $h_2\left(n\right)$. These data were calculated using PARI/GP \cite{PARI2}.\label{tab:p-vals-of-H2}}
\end{table}

The $5$-valuation is $2$ from $n=4$ to $n=2,499$. The $5$-valuation is then $0$ from $n=2,500$ to $n=2,999$. The 5-valuation is then 1 from $n=3,000$ to 
$n=12,499$, and then the $5$-valuation remains at $2$ through the end of the
run (at $n=27,477$).

Of the 43 remaining primes less than $200$, all of them enter with non-zero $p$-valuations as $n$ grows. For some primes the first non-zero valuation is $1$ and in other cases the first non-zero $p$-valuation is $2$. For example, up to $n=6$, the 7-valuation is $0$; beginning with $n=6$, the $7$-valuation is $2$ through the end of the run. Only the $11$-valuation exceeds $2$ in the run; at $n=848$, the $11$-valuation of $h_{2}\left( n\right) $ becomes $3 $ and remains $3$ through the end of the run.

With a single exception among the primes between 7 and 200, once the prime acquires a non-zero $p$-valuation, the $p$-valuation does not decline. At $n=9,323$, the $97$-valuation returns to $1$ (from $2$), and beginning at $n=9,408$, the $97$-valuation falls to 0 where it remains through the remainder of the run. Thus, from $n=9,408$ through the end of the run ($n=27,477$), $2$ and $97$ are the only primes less than $200$ for which in the denominator of $h_{2}\left( n\right) $ has a $p$-valuation of zero. The behavior of the $97$-valuation is tabulated in right panel of Table~\ref{tab:p-vals-of-H2}.

Our best guess is that all the odd $h_{j}$ have similar behavior so that the conjecture will be true for the $j$th iterated harmonic numbers and provable by showing the $2$-valuation of the denominators to be positive. We also guess that the conjecture will hold for all the even $h_{j}$, but that the proof will be quite difficult.

\section{\mbox{The hyperharmonic numbers of Conway and Guy}}\label{sec:conway-guy-1996}

As our last topic, we compare our iterated harmonic numbers to the hyperharmonic numbers of Conway and Guy \cite{conway-guy-1996},  a different set of sequence generalizing the harmonic numbers.  In \emph{The Book of Numbers}, Conway and Guy generalize the harmonic numbers to the hyperharmonic numbers \cite{conway-guy-1996}.
 In their notation, the sequence of harmonic numbers are designated as $H_n^{(1)}$, the sequence of second harmonics is defined by 
 $H_n^{(2)}= H_1^{(1)}+H_2^{(1)}+\cdots+ H_n^{(1)}$, 
the sequence of third harmonics is defined by $H_n^{(3)}= H_1^{(2)}+H_2^{(2)}+\cdots+ H_n^{(2)}$ , and so on.

\cite{conway-guy-1996} presents the hyperharmonic numbers without motivating their existence. A natural motivation comes from summability of infinite series. We look at Ernest Cesàro's sequence of summation methods $(C,k),k=0,1,2,\dots$. (See section 5.4 of \cite{hardy-divergent-1973}.) Let $\Sigma = \sum_{i=1}^\infty{a_i}$ be an infinite series. Let $A_n^{(1)}$ be the sequence of partial sums of the series, let  $A_n^{(2)}= A_1^{(1)}+A_2^{(1)}+\cdots+A_n^{(1)}$, and  $A_n^{(3)}=A_1^{(2)}+A_2^{(2)}+\cdots+ A_n^{(2)}$, and so on. 

Obviously the hyperharmonic number $H_n^{(k)}$ is exactly the number $A_n^{(k)}$ for the special case where the original series is specialized to $a_i = \frac1i$ for all $i$. However it seems like the motivation goes no deeper. The next paragraph sketches how the $A_n^{(k)}$ fit into summability theory. 

When $\lim_{n\to\infty}A_n^{(1)} = A^{(1)}$, say that $\Sigma$  is $(C,0)$ summable to $ A^{(1)}$, so  $(C,0)$ summability is ordinary convergence. When $\lim_{n\to\infty}\frac{A_n^{(2)}}{ n} = A^{(2)}$, say that $\Sigma$  is $(C,1)$ summable to $ A^{(2)}$. When $\lim_{n\to\infty}\frac{A_n^{(3)}}{\binom{n+1}{2}} = A^{(3)}$, say that $\Sigma$  is $(C,2)$ summable to $ A^{(3)}$.  For $ i<j, A^{(i)}$ exists implies $A^{(j)}$ exists and  $A^{(j)} =  A^{(i)}$. The series $1-1+1-1+\cdots$ is not  $(C,0)$ summable, but is   $(C,1)$ summable to $\frac{1}{2}$. The series $1-2+3-4+\cdots$ is not  $(C,1)$ summable, but is  $(C,2)$ summable to $\frac{1}{4}$.  

During the twenty year period after the creation of the hyperharmonic numbers, a lot of evidence, both numerical and theoretical, was piling up in support of the conjecture that 1 was the only hyperharmonic integer. However, in 2020, some very large hyperharmonic integers were revealed, the smallest of which is $H_{33}^{( 64(2^{2659}-1)+32)}$  \cite{Sertbas-2020}.
 
\section{Musings}\label{sec:musings}
 
For each $j$, the sequence $\{h_j(n)\}$ is increasing in $n$. For the harmonic numbers and for $j=2$ the sequences are concave, i.e., $h_j(n+2) - 2h_j(n+1)+h_j (n) <0$ for $n=1,2,\ldots$ It might be interesting to check monotonicity in $j$ (for fixed $n$) and concavity of $\{h_j(n)\}$ for $j>2$. Another interesting, accessible problem would be to prove the inequalities in Equation~\ref{eq:inequality}.

For each $j$, our iterated harmonic numbers $h_j(\cdot)$ and the Conway--Guy hyperharmonic numbers $H_\cdot^{(j)}$ are sequences of rational numbers tending to infinity. When $j=1$, both generalizations coincide with the harmonic numbers and hence contain no integers larger than $1$. The naïve guess that the ever more slowly increasing sequences $h_j$ are more likely to intersect the integers than are the much more rapidly increasing $H^{(j)}$---and that this effect increases with increasing $j$---is revealed as overly simplistic by our computations and Sertba\c{s}' example. With respect to the Conjecture~\ref{conjecture:no-integers}, we have a feeling that the odd iteration cases all involve the $2$-valuation and may be more like the $j=1$ classical case and the even iteration cases may all be similar to the $j=2$ case. In particular, $j=3$ may be the easiest of all the open cases.

Comparing Theorems~\ref{th:gamma-j-prime-ok} and \ref{th:gamma-j-not-ok} shows that the $\gamma_{j}$ can easily be estimated to several decimal places, while the 
$\gamma _{j}^{\prime}$ cannot. That is, although $h_{j}-\gamma _{j}^{\prime}$ is indeed an estimator for the $j$th iterated logarithm, it is not nearly as good as is $l_{j}-\gamma _{j}$. The $h_{j}$ are rational but do not give practical approximations for iterated logarithms. Even for $\gamma_2$ and $\gamma_3$, approximation to two significant figures would likely require a new idea.

An infinite list of difficult open questions asks whether the $\gamma _{j}$ and $\gamma _{j}^{\prime }$ are transcendental or, at least, irrational. This is a well known open question when $j=1$ and $\gamma _{1}=\gamma _{1}^{\prime }=\gamma $, where $\gamma =.577\ldots$ is Euler's constant. For each $j \geq 2$, the value of the $\gamma_j$
appearing in Theorem 2 depended on an arbitrary choice of a lower limit of an integration. If someone could determine a canonical or natural way of making that choice, then the
questions about $\gamma_j$ for $j \geq 2$ would become more important. 

We have not found any direct generalizations of the harmonic numbers other than the Conway-Guy hyperharmonic numbers and our own iterated harmonic numbers.  It would be interesting to see if there have been other generalizations. One place to look could be among various generalizations of Euler's constant~$\gamma $ associated with Stieltjes.

  We thank Stefan Catoiu, Gang Wang and anonymous referees for suggestions that improved the paper.

\bibliography{IteratedHarmonicNumbers}

\providecommand{\bysame}{\leavevmode\hbox to3em{\hrulefill}\thinspace}
\providecommand{\MR}{\relax\ifhmode\unskip\space\fi MR }
\providecommand{\MRhref}[2]{%
  \href{http://www.ams.org/mathscinet-getitem?mr=#1}{#2}
}
\providecommand{\href}[2]{#2}
\begin{thebibliography}{10}

\bibitem{AZ}
M~Aigner and GM~Ziegler, \emph{{Proofs from The Book, 6th ed.}}, Berlin,
  Heidelberg, 2018.

\bibitem{apostol-calculusii-1969}
TM~Apostol, \emph{{Calculus, Volume II}}, 2nd ed., Waltham: Blaisdell
  Publishing Company, 1969.

\bibitem{ash-iterated-logarithms-2009}
JM~Ash, \emph{{Series involving iterated logarithms}}, {College Math. J.}
  \textbf{40} (2009), 40--42.

\bibitem{ash-plaza-2021}
JM~Ash and Á~Plaza, \emph{{$\sum_{n=2}^{\infty }\frac{1}{nH_{n-1}}$ diverges
  while $\sum_{n=2}^{\infty }\frac{1}{n\left(H_{n}\right) ^{1+\epsilon }}$
  converges}}, The Mathematical Gazette (2021), 161--162.

\bibitem{conway-guy-1996}
JH~Conway and R~Guy, \emph{{The Book of Numbers}}, New York: Springer-Verlag,
  1996.

\bibitem{Gourdon-Sebah}
X~Gourdon and P~Sebah, \emph{{The Euler constant: $\gamma$}}, 2004, available
  from \url{http://numbers.computation.free.fr/Constants/Gamma/gamma.html}.

\bibitem{hardy-divergent-1973}
GH~Hardy, \emph{{Divergent Series}}, Oxford: Oxford Univ. Press, 1973.

\bibitem{hardy-wright-number-1960}
GH~Hardy and EM~Wright, \emph{{An Introduction to the Theory of Numbers, 4th
  ed.}}, Oxford: Oxford University Press, 1960.

\bibitem{Kurschak}
J~K{\"u}rsch{\'a}k, \emph{{A Harmonikus Sorr{\'o}l}}, {Mat. {\'e}s Fiz.
  Lap\'ok} \textbf{27} (1918), 299--300.

\bibitem{Sertbas-2020}
DC~Sertba\c{s}, \emph{{Hyperharmonic integers exist}}, Comptes Rendus
  Mathématique \textbf{358} (2020), no.~11-12, 1179--1185.

\bibitem{PARI2}
{{The PARI~Group}}, \emph{{PARI/GP version \texttt{2.11.2}}}, Univ. Bordeaux,
  2019, available from \url{http://pari.math.u-bordeaux.fr/}.

\bibitem{theisinger-bemerkung-1915}
L~Theisinger, \emph{{Bemerkung {\"u}ber die harmonische Reihe}}, {Monatshefte
  f{\"u}r Mathematik und Physik} \textbf{26} (1915), no.~1, 132--134.

\end{thebibliography}

\end{document}